\documentclass[11pt]{amsart}

\usepackage[T1]{fontenc}
\usepackage{amsfonts,amssymb,amsthm, txfonts, pxfonts,amscd}
\usepackage[stretch=10]{microtype}
\usepackage{epsfig}
\newtheorem{thm}{Theorem}
\newtheorem{prop}[thm]{Proposition}
\newtheorem{lem}[thm]{Lemma}
\newtheorem{cor}[thm]{Corollary}

\theoremstyle{definition}

\theoremstyle{remark}
\newtheorem{remark}{Remark}

\newcommand{\Z}{\mathbb{Z}}

\newcommand{\zz}[1]{\mathbb #1}
\newcommand{\xclass}[1]{\langle #1 \rangle}

\begin{document}

\title[Maximal Displacement of  Critical Branching Random Walk]{On the
Maximal Displacement of a Critical Branching Random Walk} 
\author{Steven P. Lalley}
\address{Department of Statistics, University of Chicago, Chicago, IL 60637}
\email{lalley@galton.uchicago.edu}
\urladdr{www.statistics.uchicago.edu/$\sim$lalley}
\author{Yuan Shao}
\address{Department of Mathematics, University of Chicago, Chicago, IL 60637}
\email{shaoyuan3319@gmail.com}
\date{\today}
\thanks{First author supported by NSF grant DMS  -
0805755}

\begin{abstract}
We consider a branching random walk initiated by a single particle at
location $0$ in which particles alternately
reproduce according to the law of a Galton-Watson process and disperse
according to the law of a driftless random walk on the integers. When
the offspring distribution has mean $1$ the branching process is
critical, and therefore dies out with probability $1$. We prove that
if the particle jump distribution has mean zero, positive finite variance
$\eta^{2}$, and finite $4+\varepsilon$ moment, and if the offspring
distribution has positive variance $\sigma^{2}$ and finite third moment then the
distribution of the rightmost position $M$ reached by a particle of
the branching random walk satisfies $P\{M \geq x\}\sim 6\eta^{2}/
(\sigma^{2}x^{2})$ as $x \rightarrow \infty$. We also prove a
conditional limit theorem for the distribution of the rightmost
particle location at time $n$ given that the process survives for $n$
generations. 
\end{abstract}

\maketitle

\section{Introduction}

It has been known since the work of McKean \cite{mckean:kpp} that the
time evolution of a one-dimensional branching Brownian motion is
intimately connected with the behavior of solutions of the Fisher-KPP
equation
\begin{equation}\label{eq:kpp}
	\frac{\partial u}{\partial t}=\frac{1}{2}
	\frac{\partial^{2}u}{\partial x^{2}} +f (u).
\end{equation}
McKean observed that the cumulative distribution function of the
position $R_{t}$ of the rightmost particle at time $t$ in a
one-dimensional branching Brownian motion obeys equation
\eqref{eq:kpp}.  In the simplest case, where particles move
independently along Brownian trajectories and undergo simple binary
fission following independent, exponentially distributed gestation
times, the function $f$ is given by $f (u)=u^{2}-u$, the case
originally studied by Fisher \cite{fisher}; more general branching
mechanisms lead to more general functional forms. In general, when the
underlying branching process is \emph{supercritical}, the solution of
\eqref{eq:kpp} with Heaviside initial data approaches a traveling wave
with a positive asymptotic velocity, and thus, in particular, the
distribution of $R_{t}$, when centered at its median $m_{t}$,
converges weakly to the distribution described by the wave. There is
now a considerable literature surrounding problems connected with this
phenomenon: see, e.g., \cite{bramson} for the precise asymptotic
behavior of the function $t\mapsto m_{t}$; \cite{lalley-sellke} for a
proof that the traveling wave is a mixture of extreme value
distributions; and \cite{aidekon-et-al}, \cite{arguin-et-al:MR3129797}
and \cite{arguin-et-al} for proofs that the distribution of the entire
point process of particle locations, when centered at $m_{t}$,
converges in law as $t \rightarrow \infty$.  Generalizations of some
of these results to supercritical branching random walks are given in
\cite{bachmann} and \cite{bramson-zeitouni}.

When the branching mechanism is \emph{critical} (that is, when the
mean number of offspring particles at a particle death is $1$), the
nature of the process $R_{t}$ is entirely different, because in this
case the process must ultimately die out, with probability one. Hence,
it is more natural in the critical case to ask about the distribution  of 
\begin{equation}
	M:=\max_{t<\infty}R_{t},
\end{equation}
the rightmost point ever reached by a particle of the branching
process. Interest in the distribution of $M$ stems in part from its
relevance to evolutionary biology, where critical branching Brownian
motion has been used as a model for the spatial diffusion of alleles
with no selective advantage or disadvantage: see, for instance,
\cite{crump-gillespie}, \cite{fleischman-sawyer}, and the references
therein.  For critical (or subcritical) branching Brownian motion the
distribution function $\omega  (x)=P\{M\leq x \}$ of $M$ satisfies the
ordinary differential equation
\begin{equation}\label{eq:fleischman-sawyer}
	\frac{1}{2} \omega '' (x)=-\psi (\omega  (x)),
\end{equation}
where $\psi (\omega )$ is the probability generating function of the
offspring distribution of the branching process. This differential
equation is similar to that satisfied by the traveling wave(s) for the
Fisher-KPP equation, but differs in the nonlinear term $\psi (\omega
)$; this leads to solutions of a very different character, reflecting
the qualitative differences between critical and supercritical
branching.  By explicitly integrating the equation
\eqref{eq:fleischman-sawyer}, Fleischman and Sawyer
\cite{fleischman-sawyer}, sec.~3, obtained precise asymptotic estimates
for the tail of the distribution of $M$ for critical branching
Brownian motion: in particular, they proved that if the offspring
distribution has mean $1$, positive variance $\sigma^{2}$, and finite third moment
then
\begin{equation}\label{eq:fleischman-sawyer-asymptotics}
	P\{M\geq x \}\sim \frac{6}{\sigma^{2}x^{2}} \quad
	\text{as}\;\; x \rightarrow \infty.
\end{equation}
(In the case of ``double-or-nothing'' branching, where particles
produce either $0$ or $2$ offspring with probability $1/2$, the
solution to the differential equation with the appropriate boundary
conditions has the closed form $1-w (x)=6/ (1+x)^{2}$, as is easily
checked.)

Our primary objective in this paper is to show that under suitable
hypotheses the asymptotic formula
\eqref{eq:fleischman-sawyer-asymptotics} holds generally for
\emph{critical, driftless branching random walks}. For branching
random walks, unlike branching Brownian motion, there are no ordinary
differential equations governing the law of $M$, so one cannot write
explicit integral formulas for the distribution of $M$, as one can for
branching Brownian motion. \footnote{To appreciate the difficulty of
generalizing from branching Brownian motion to branching random walk,
one should compare with the recent history of progress on the
supercritical case. For supercritical branching Brownian motion the
law of the maximal displacement $R_{t}$ is governed by the KPP
equation, and so the weak convergence of $R_{t}-\text{med} (R_{t})$
follows by monotonicity and comparison techniques for KPP-type
parabolic equations, as has been understood since the 1970s; but for
supercritical branching random walks, completely new techniques
\cite{bramson-zeitouni}  were required to
prove even that the distributions of the maximal displacements after
$n$ steps are tight.}

 For ease of
exposition we will limit our study to discrete-time branching random
walks on the integer lattice $\zz{Z}$ in which the reproduction and
dispersal mechanisms are independent, with dispersal \emph{preceding}
reproduction. Thus, in each generation, particles
\begin{enumerate}
\item [(A)]  jump
(independently of one another) to new sites, with jumps following a
\emph{mean-zero}, \emph{finite variance}, jump distribution
$F_{RW}:=\{a_{x} \}_{x\in \zz{Z}}$; 
\item [(B)] then reproduce as in a simple Galton-Watson process,
according to a fixed offspring distribution $F_{GW}:=\{p_{k} \}_{k\geq
0}$ with mean $1$ and finite variance.
\end{enumerate}
We will generally assume that the branching random walk is initiated
by a single particle located at the origin $0\in \zz{Z}$. For a formal
construction of a branching random walk following rules (A)--(B), see,
for instance, \cite{kesten}. The locations of the particles in
generation $n$ will be denoted by $X_{n,i}$, where $i\leq N_{n}$ and
$N_{n}$ is the total number of  particles in the $n$th generation. 
Our interest is in the distribution of the \emph{maximal displacement}
\begin{equation}\label{eq:max-displacement}
M=\max_{n \geq 0} M_{n} \quad \text{where} \;\; M_{n}=
\max_{i=1,2,\cdots,N_n} X_{n,i}.
\end{equation}

\begin{remark} \label{rmk}
Observe that only the particle locations at the end of each
reproduction step are taken into account in the definition of $M$:
thus, for instance, if the initial particle at $X_{0,1}=0$ were to
jump to site $x=1$ and then produce no offspring, the maximal
displacement would be $M=0$, not $M=1$.
\end{remark}

\begin{remark}\label{rmk:2}
Many authors (e.g., \cite{kesten}, \cite{benjamini-peres}) consider
branching random walks in which the order of the reproduction and jump
steps is opposite to that specified above. For our purposes it is more
convenient to have the jump step precede the reproduction step, as
this leads to more compact formulations of the main result and the
nonlinear convolution equation \eqref{fund} below. It should be
clear that analogous results can be deduced for branching random
walks of the type considered in \cite{kesten}, \cite{benjamini-peres},
by conditioning on the first jump step. See
Remark~\ref{remark:alternate-steps} in sec.~\ref{sec:conv-eqn} below.
Although it is less obvious, our results (and their proofs) also
generalize to branching random walks in which the reproduction and
particle motion are not independent (as, for instance, in the
branching random walks discussed in \cite{biggins}), and to branching
random walks in which the step distribution is supported by $\zz{R}$
rather than $\zz{Z}$. Thus, in particular, our results extend to
critical branching Brownian motion, since branching Brownian motion
observed at integer multiples of a fixed time $\delta >0$ is a
branching random walk. In the interest of clarity,  we have
chosen not to formulate and prove our results in the greatest possible
generality.
\end{remark}

\begin{thm} \label{mainresultgeneral}
Assume that the step distribution $F_{RW}=\{a_k\}_{k \in \Z}$  has
mean $0$, positive variance $\eta^{2}$, and  finite $r$-th moment for some
$r>4$.  Assume also that the offspring distribution $F_{GW}=\{p_{k}
\}_{k\geq 0}$ has mean $1$, positive variance $\sigma^{2}$, and finite third
moment. Then
\begin{equation}\label{eq:main-result}
	P\{M\geq x \}\sim \frac{6\eta^{2}}{\sigma^{2}x^{2}} \quad
	\text{as}\;\; x \rightarrow \infty.
\end{equation}
\end{thm}

This will be proved in section~\ref{sec:main-result}. The result can
also be reformulated as a statement about the distribution of the
maximal displacement of a branching random walk initiated by a large
number $n$ of particles at the origin. Such a branching random walk is
just the superposition (sum) of $n$ independent copies of
the branching random walk in Theorem~\ref{mainresultgeneral}, so the
event that its maximal displacement is $\leq \sqrt{n}x$ is the intersection of
the events that each of the $n$ constituent branching random walks has
maximal displacement $\leq \sqrt{n}x$. For fixed $x>0$ the target
points $\sqrt{n}x \rightarrow \infty$ as $n \rightarrow \infty$, so
the asymptotic formula \eqref{eq:main-result} yields the following
corollary.

\begin{cor}\label{corollary:max-displacement-superposition}
Let $M^{n}$ be the maximal displacement of a branching random walk
initiated by $n$ particles at the origin at time $0$. Then
for any $x>0$,
\begin{equation}\label{eq:max-displacement-superposition}
	\lim_{n \rightarrow \infty}P\{M^{n} \geq
	\sqrt{n}x\}= 1-\exp \{-C/x^{2} \} 
	\quad \text{where} \quad C=\frac{6\eta^{2}}{\sigma^{2}}.
\end{equation}
\end{cor}

Theorem~\ref{mainresultgeneral} is closely related to the main results
of Kesten \cite{kesten}, who considers critical branching random walk
conditioned to survive for a large number of generations. (X. Zheng
\cite{zheng} has also considered the case where the random walk is
assumed to have a ``small'' drift, but this leads to entirely
different asymptotics. Aidekon \cite{aidekon:killed} considers
branching random walks that are rendered critical by virtue of killing
at the origin: here again the asymptotic behavior of the maximum behaves
quite differently.)  Kesten shows that under the same hypotheses as in
Theorem~\ref{mainresultgeneral}, for any fixed $\beta >0$, given that
the branching process survives for $\beta n$ generations, the
conditional distribution of $\max_{k\leq n}M_{k}/\sqrt{n}$ converges
as $n \rightarrow \infty$. He does not, however, identify the limit
distribution. He remarks on the result of Sawyer and
Fleischman:

\begin{quote}
We have not proved [equation \eqref{eq:main-result}] in our
setting. It is not clear at the moment how the methods of Sawyer and
Fleischman, which rely on differential equations, can be carried over
to the discrete setting; differential equations will have to be
replaced by recurrence relations.
\end{quote}
 
\noindent 
Our main technical innovation  will be to show how to
relate these ``recurrence relations'' to the differential equation
\eqref{eq:fleischman-sawyer}. This will be accomplished by exploiting
Feynman-Kac formulas. Our approach applies also to time-dependent
Feynman-Kac problems, and leads in particular to information about the
distribution of the random variable $M_{n}$. As an illustration, we
will prove in section~\ref{sec:time-dependent} the following
conditional limit theorem.

\begin{thm}\label{theorem:conditional}
Under the hypotheses of Theorem~\ref{mainresultgeneral}, the
conditional distribution of $M_{n}/\sqrt{n}$, given that the branching
process survives for $n$ generations, converges weakly as $n
\rightarrow \infty$ to a nontrivial limit distribution $G$ that
depends only on the variances $\sigma^{2}$ and $\eta^{2}$ of the
offspring and step distributions.
\end{thm}

The scaling in this theorem is the same as that in the Dawson-Watanabe
theorem (see, for instance, \cite{etheridge}, ch.~1), which can be
stated as follows. Suppose that $n$ independent copies of the
branching random walk are initiated at the origin $0\in \zz{Z}$. If
particles are assigned mass $1/n$, and if time and space are re-scaled
by factors $1/n$ and $1/\sqrt{n}$, respectively, then the
corresponding measure-valued processes $BRW_{n}(t)$ converge weakly to
\emph{super-Brownian motion} $X_{t}$.  A similar theorem holds for the
measure-valued processes $BRW^{*}_{n} (t)$ attached to branching random walk
initiated by a single particle at the origin, but conditioned to
survive for $n$ generations: under the same re-scaling of mass, time,
and space as in the Dawson-Watanabe theorem, the measure-valued
processes $BRW^{*}_{n} (t)$ converge weakly as $n \rightarrow \infty$ to a
measure-valued process $Y_{t}$. The law of this process $Y_{t}$ is
related to that of the super-Brownian motion by the \emph{Poisson
cluster representation}: if $N$ is a Poisson random variable with
mean $1$ and $Y^{1}_{t},Y^{2}_{t},\dotsc$ are independent copies of $Y
_{t}$ then $\sum_{i=1}^{N}Y^{i} _{t}$ is a super-Brownian
motion. 

The weak convergence of the measure-valued processes $BRW^{*}_{n} (t)$
by itself does not imply Theorem~\ref{theorem:conditional}, because
the location $M_{n}/\sqrt{n}$ is not a continuous function (relative
to the weak topology on measures) of $BRW^{*}_{n}
(1)$. (Lalley\cite{lalley:spatial} shows that in dimension $1$
rescaled branching random walks converge to super-Brownian motion in a
stronger topology than the weak topology implicit in the
Dawson-Watanabe theorem. However, even this topology is too weak to
make the normalized rightmost particle location a continuous
functional.) Nevertheless, it is natural to wonder whether how the
limit distribution $G$ of Theorem~\ref{theorem:conditional} is related
to the limiting measure-valued process $Y_{t}$. The proof of
Theorem~\ref{theorem:conditional} will establish that $G$ is the
distribution of the rightmost support point of the random measure
$Y_{1}$.

\begin{cor}\label{corollary:G}		
Under the hypotheses of Theorem~\ref{mainresultgeneral} (in
particular, under the assumption that the step distribution $F_{RW}$
has finite $r$th  moment for some $r>4$),
\begin{equation}\label{eq:definition-G}
	G (x) = P\{ Y_{1} [x,\infty)=0 \}.
\end{equation}
\end{cor}

Are $4+\varepsilon$ moments on the step distribution really necessary
for the validity of our theorems (and Kesten's)? The following simple
heuristic argument (which with a bit of work can be made rigorous)
shows that $4-\varepsilon$ moments are not enough. Consider, for
instance, the case where $F_{RW}$ is the symmetric distribution on the
nonzero integers with discrete density
\[
	f_{RW} (x)=\frac{1}{2\zeta (5-\varepsilon )|x|^{5-\varepsilon }},
\]
which has infinite $4$th moment. Conditional on the event that the
branching random walk survives for at least $n$ generations it will
produce on the order of $n^{2}$ particles. Each of these has
conditional probability $\sim C/n^{2-\alpha }$ of placing an
offspring at distance $n^{(1+\delta)/2}$ to the right, where 
$\alpha =\varepsilon /2+\varepsilon \delta -4\delta$.
Consequently, if $\alpha >0$ then for large $n$ the probability that
all $n^{2}$ particles are located in an interval $[-A\sqrt{n},A\sqrt{n}]$
is vanishingly small.

\section{Maximal Displacement: Proof of
Theorem~\ref{mainresultgeneral}}\label{sec:main-result}

\subsection{A Nonlinear Convolution Equation}\label{sec:conv-eqn}

For the remainder of the paper we shall assume that the offspring
distribution $F_{GW}=\{p_{k}\}_{k\geq 0}$ and the jump distribution
$F_{RW}=\{a_{x} \}_{x\in \zz{Z}}$ satisfy the hypotheses of
Theorem~\ref{mainresultgeneral}: in particular, $F_{GW}$ has mean $1$,
positive variance $\sigma^{2}$, and finite third moment, and $F_{RW}$
has mean $0$, positive variance $\eta^{2}$, and finite $4+\varepsilon$
moment. The maximal displacement $M$ of the branching random walk is
defined by equation~\eqref{eq:max-displacement}, and its (tail)
distribution function will be denoted by
\begin{equation}\label{eq:tail-d-f}
	u(x) = P\{M> x \}.
\end{equation}
Clearly, $u$ is non-increasing in $x$, with jump discontinuities at
integer arguments. Also, $u (x)=1$ for all $x\leq 0$, and $\lim_{x \rightarrow
\infty}u (x)=0$. 

We begin by showing that $u$ satisfies a nonlinear convolution
equation analogous to the Fleischman-Sawyer
equation~\eqref{eq:fleischman-sawyer}. This is obtained in the
conventional manner, by conditioning on the first generation of the
branching process. Each particle of the first generation will give
rise to its own descendant branching random walk, independent of its
siblings (conditional on their locations), and in order that $M\leq x$
it must be the case that the maximal displacements of all of the
descendant BRWs, adjusted by their starting points, must be $\leq x$.
This implies that for all $x\geq 1$,
\begin{equation}\label{eq:convolution-naive}
	1-u (x) = \sum_{y\in \zz{Z}}a_{y} \sum_{k=1}^{\infty}p_{k}
	(1-u (x-y))^{k} .
\end{equation}
(Note that this is consistent with our convention regarding the
sequencing of the dispersal and reproduction steps -- see
Remark~\ref{rmk} above.) Rewriting this equation  in terms of $u$ leads
immediately to the following proposition.

\begin{prop}
$u(x)$ satisfies the nonlinear convolution equation
\begin{equation} \label{fund}
u(x) = \sum_{y \in \Z} a_y Q(u(x-y)),
\end{equation}
where $1-Q (1-s)$ is the probability generating function of the
offspring distribution $F_{GW}$, that is,
\begin{equation}\label{eq:definition-Q}
	Q(s) = 1-\sum_{i=0}^{\infty} p_i (1-s)^i, \qquad \text{for}
	\;\; 0 \leq s \leq 1. 
\end{equation}
\end{prop}

\begin{remark}
\label{remark:alternate-steps} If the branching random walk used the
alternative rule discussed in Remark $\ref{rmk}$ (that is, particles
first reproduce and then disperse) to construct the branching random
walk, then equation \eqref{fund} would change as follows. Writing
$\tilde{M}$ for the maximal displacement of this branching random
walk, and $\tilde{u}(x) = P \{ \tilde{M} \geq x \}$, we would have
$$\tilde{u}(x) = Q \Bigl ( \sum_{k \in \Z} a_k \tilde{u}(x-k) \Bigr
).$$ Comparing this with equation $\eqref{fund}$, we see
that $$\tilde{u}(x) = Q \bigl ( u(x) \bigr ).$$ Since the Taylor
expansion of $Q$ is $Q(s) = s -{\sigma^2}s^2/{2} + O(s^3)$, it
follows that $\tilde{u} (x)$ and $u(x)$ go to $0$ as $x \rightarrow
\infty$ at the same rate.
\end{remark}

\subsection{A Discrete Feynman-Kac Formula}\label{sec:fk}

Our goal now is to analyze the asymptotic behavior of solutions to the
nonlinear convolution equation \eqref{fund} as $x \rightarrow
\infty$. To accomplish this, we will show that solutions of
\eqref{fund} can be represented by formulas of ``Feynman-Kac'' type.
Henceforth, we shall denote by $W_{n}$ the random walk on $\zz{Z}$
whose step distribution is the reflection of the step distribution
$F_{RW}$ in the underlying BRW , that is,
\begin{equation}\label{eq:rw-law}
	P (W_{n+1}-W_{n}=y\,|\, W_{n},W_{n-1},\dotsc)=a_{-y}.
\end{equation}
We shall use superscripts  $P^{x}$ and $E^{x}$ to denote the initial
point $W_{0}=x$ of the random walk $W_{n}$.

Define 
\begin{align}\label{eq:h-H}
	h(s) &= s - Q(s) = {\sigma^2 s^{2}/2}+ O(s^3) \quad \text{and}\\
\notag 	H(s) &= {h(s)/s} = {\sigma^2 s/2} + O(s^2).
\end{align}
It is easily checked that $H (s)$ is increasing for $s\in [0,1]$, and satisfies
$H (0)=0$ and $H (1)=p_{0}$.

\begin{prop} \label{ynmtg}
Under $P^{x}$, the process 
\begin{equation}\label{eq:fk-mg}
	 Y_n = \Bigl ( \prod_{j=1}^n \bigl (1-H(u(W_j)) \bigr ) \Bigr ) \cdot u(W_n)
\end{equation}
is a bounded martingale with respect to the natural filtration
generated by the random walk  $\{ W_n \}$. Here we adopt the convention
that the empty product $\prod_{j=1}^0$ is equal to $1$. 
\end{prop}

\begin{proof}
The random variables 
$Y_n $ are uniformly bounded, in particular, $0 \leq Y_n \leq 1$. To prove that the
sequence is a martingale we appeal to the nonlinear convolution
equation \eqref{fund}, which can be rewritten in terms of the function
$h$ as
\begin{equation} \label{fund2}
\left( \sum_{k \in \Z} a_k u(x-k)  \right) - u(x) = \sum_{k \in \Z} a_k h(u(x-k)).
\end{equation}
Using this, we compute
\begin{equation*}
\begin{aligned}
& E^x \Bigl ( Y_{n+1} \mid \{ W_j \} _{j=1}^n \Bigr )\\
= \; & \Bigl ( \prod_{j=1}^n \bigl (1-H(u(W_j)) \bigr ) \Bigr ) \cdot E^x \Bigl ( \bigl (1-H(u(W_{n+1})) \bigr ) \cdot u(W_{n+1}) \mid \{ W_j \} _{j=1}^n \Bigr )\\
= \; & \Bigl ( \prod_{j=1}^n \bigl (1-H(u(W_j)) \bigr ) \Bigr ) \cdot E^x \Bigl ( \bigl (u(W_{n+1}) - h(u(W_{n+1})) \bigr ) \mid \{ W_j \} _{j=1}^n \Bigr )\\
= \; & \Bigl ( \prod_{j=1}^n \bigl (1-H(u(W_j)) \bigr ) \Bigr ) \cdot \Bigl ( \sum_{k \in \Z} a_k u(W_n - k) - \sum_{k \in \Z} a_k h(u(W_n - k)) \Bigr )\\
= \; & \Bigl ( \prod_{j=1}^n \bigl (1-H(u(W_j)) \bigr ) \Bigr ) \cdot u(W_n)
= \;  Y_n,
\end{aligned}
\end{equation*}
where the second to last equality uses  equation $\eqref{fund2}$.
\end{proof}

\begin{cor}\label{corollary:fk}
For each $y\in \zz{Z}$, define 
\begin{equation}\label{eq:tau}
\tau_{y}=\min \{ n \geq 0 \mid W_n \leq y \}.
\end{equation}
Then for all $x,y\in \zz{Z}$,
\begin{equation} \label{mtg}
u(x) = E^x \left( \prod_{j=1}^{\tau_{y}} \bigl (1-H(u(W_j)) \bigr ) \right)u (W_{\tau_{y}}).
\end{equation} 
\end{cor}

\begin{proof}
Since the random walk $W_{n}$ is driftless it must be recurrent, and
hence $\tau_{y}$ is finite.  Since the martingale $Y_{n}$ of
Proposition~\ref{ynmtg} is bounded, Doob's optional sampling identity
applies, yielding \eqref{mtg}. Note that if $y\leq 0$ then $u
(W_{\tau_{y}})=1$, since $W_{\tau_{y}}\leq 0$.
\end{proof}

\subsection{Scaling Limits}\label{sec:scaling}

Using the Feynman-Kac representation \eqref{mtg} we will show that the
function $u$, properly re-normalized, converges to a function that
satisfies the Fleischman-Sawyer equation \eqref{eq:fleischman-sawyer}.
Because we do not know \emph{a priori} that the function $u$ has a
proper scaling limit we must work with subsequential limits. 
Since $u$ is monotone and satisfies $0<u\leq 1$ it follows that for
any $y\geq 0$ there exist sequences $x_{k} \rightarrow \infty$ such
that
\begin{equation}\label{eq:subseq-lims}
	\phi (y):= \lim_{ k \rightarrow \infty} \frac{u
	(x_{k}+{y}/{\sqrt{u(x_{k})}}  )}{u(x_{k})}  
\end{equation}
exists.  Clearly, any such limit must satisfy $0\leq \phi (y)\leq 1$, and
if $y=0$ the limit is $\phi (0)=1$. By Cantor's diagonalization
argument, any such sequence $x_{k}$ must have a subsequence, which we
also denote by $x_{k}$, such that the convergence
\eqref{eq:subseq-lims} holds for all \emph{rational} $y\geq 0$.

\begin{prop} \label{phicont} For any sequence $x_{k}\rightarrow
\infty$ such that \eqref{eq:subseq-lims} holds for all \emph{rational}
$y\geq 0$, the limit function $\phi (y)$ extends to a continuous,
non-increasing, positive function of $y\in [0,\infty )$. Hence, the
convergence \eqref{eq:subseq-lims} holds {uniformly} for $y$ in any
compact interval $[0,A]$.
\end{prop}

\begin{proof}
Fix $0\leq y_1<y_2$, both rational, and for ease of notation write
$z_i={y_i}/{\sqrt{u(x)}}$ for $i=1,2$ and $x=x_{k}$ (the dependence on
$k$ will be suppressed). Fix a sequence
$x_{k}\rightarrow \infty$ along which \eqref{eq:subseq-lims} holds for
all rational $y$. To avoid a proliferation of subscripts, we shall
write $\lim_{x \rightarrow \infty}$ to mean convergence along the
subsequence $x_{k}$.  By Proposition \ref{ynmtg} and Doob's optional
sampling theorem,
\begin{equation*}
\begin{aligned}
\phi (y_{2})
& = \lim_{x \rightarrow \infty} \frac{u (x+z_{2})}{u(x)}\\
& = \lim_{x \rightarrow \infty} E^{x+z_2}\left(
\frac{u(W_{\tau(x+z_{1})})}{u(x)} \prod_{j=1}^{\tau ( x+z_1)}  
(1-H(u(W_j))  )\right),
\end{aligned}
\end{equation*}
where $\tau_{z} = \tau (z)=\min \{ j \geq 0 \mid W_j \leq z \}$. Using
the expansion $H (u)\sim \sigma^{2}u/2$ as $u \rightarrow 0$ we obtain
\begin{equation*}
\begin{aligned}
	\phi (y_{2})=& \lim_{x \rightarrow \infty} E^{x+z_2-z_1}
	\left ( \prod_{j=1}^{\tau (x)} \left (1-H(u(W_j+z_1)) \right )
	\frac{u(W_{\tau (x)}+z_1)}{u(x)} \right ) \\
	 = \; & \lim_{x \rightarrow \infty} E^{x+z_2-z_1} \left (
	 \exp \left \{  -\frac{\sigma^2}{2} \sum_{j=1}^{\tau (x)}
	 u(W_j+z_1) \right \} \frac{u(W_{\tau (x)}+z_1)}{u(x)} \right ) \\
	  \geq \; & \lim_{x \rightarrow \infty} E^{x+z_2-z_1} \left
	  ( \exp \left \{ -\frac{\sigma^2}{2} \tau (x) u(x) \right \}
	  \frac{u(x+z_1)}{u(x)} \right ). 
\end{aligned}
\end{equation*}
The last inequality holds because $W_j > x$ for all $j < \tau_x$ and
$W_{\tau_x} \leq x$. 

By the invariance principle (here we use the
assumption that the step distribution of the random walk $W_{n}$ has
mean $0$ and finite variance), as $x \rightarrow \infty$,
\begin{equation*}
\begin{aligned}
\mathcal{D} (\tau_x u(x) \mid W_0 = x+z_{2}-z_{1}) & =
\mathcal{D} (\tau_0 u(x) \mid W_0 = z_{2}-z_{1}) \\
& \Longrightarrow  \mathcal{D} (\tau_0^{BM} \mid B_0 = y_2-y_1)
\end{aligned}
\end{equation*}
where $B_t$ is a standard Brownian motion started at $y_{2}-y_{1}$ and $\tau_0^{BM}$ is the
first hitting time of $0$ by $B_t$. (Here $\mathcal{D}$ denotes 
conditional distribution.) Hence, for any $\epsilon>0$, when
$y_2-y_1$ is sufficiently small and $x$ is sufficiently large, $\exp
 \{-\sigma^{2}\tau_{x}u (x)/2 \} \geq 1-\epsilon$ with
probability $1-\epsilon$. Consequently,
\begin{equation*}
\begin{aligned}
 \lim_{x \rightarrow \infty} &E^{x+z_2-z_1}\left ( \exp \bigl ( -\frac{\sigma^2}{2} \tau_x u(x) \bigr ) \frac{u(x+z_1)}{u(x)} \right )\\ 
\geq \; & (1-\epsilon)^2 \lim_{x \rightarrow \infty} \frac{u(x+z_1)}{u(x)}\\
= \; & (1-\epsilon)^2 \phi(y_1).
\end{aligned}
\end{equation*}
This proves that $\forall \epsilon>0$, if $z_2-z_1$ is sufficiently
close to zero then $\phi(y_1) \geq \phi(y_2) \geq (1-\epsilon)^2
\phi(y_1)$. Therefore, $\phi$ is continuous.
\end{proof}

Our aim now is to show that there is only one possible subsequential
limit function $\phi$, and that it satisfies the Fleischman-Sawyer
differential equation. To accomplish this we will use the discrete
Feynman-Kac formula \eqref{mtg} and the invariance principle to 
show that any subsequential limit $\phi$  satisfies the following
Feynman-Kac formula. 

\begin{prop} \label{phieq}
Assume that the step distribution $\{a_k\}_{k \in \Z}$ of the random
walk $\{W_n\}$ has finite $r$-th moment for some $r>4$. Then any
subsequential limit  $\phi(y)$ specified by \eqref{eq:subseq-lims}
satisfies 
\begin{equation} \label{phieqeq}
\phi(y) = E^{y/\eta} \exp \bigl ( -\frac{\sigma^2}{2}
\int_0^{\tau_0^{BM}} \phi(\eta B_t) \,\mathrm{d}t \bigr )  \quad
\text{for} \; y\geq 0, 
\end{equation}
where under $P^{y\eta}$ the process $B_{t}$ is a standard Brownian
motion started at $B_{0}=y/\eta$ and $\tau_0^{BM}$ is the first
hitting time of $0$ by $B_t$.
\end{prop}

The need for a finite $4+\varepsilon$ moment stems from the fact that
in general the random walk $W_{n}$ will overshoot $0$ at the first
passage time $\tau (0)$.  The renewal theorem (cf. \cite{feller} or
\cite{spitzer}) implies that as the initial point $W_{0}=x \rightarrow
\infty$, the distribution of the overshoot $W_{\tau (0)}$ converges
weakly provided the step distribution of the random walk has mean zero
and finite variance. The number of finite moments of the limiting
overshoot distribution is determined by the number of moments of the
step distribution, as follows.

\begin{lem} \label{moments}
 If the step distribution of 
$\{W_n\}$ has finite $r$-th moment, then the limiting overshoot
distribution  has finite $(r-2)$-th
moment. 
\end{lem}

\begin{proof}
Consider the ladder variables
\begin{align*}
	T_1 &= \min \{ n > 0 \mid W_n <W_{0} \}, \hspace{0.5cm}
	&Z_1=W_{T_1}-W_{0},\\
	 T_2 &= \min \{ n > T_1 \mid W_n <
	W_{T_{1}} \}, \hspace{0.5cm} &Z_2=W_{T_2} -W_{T_{1}},  
\end{align*}
and so
on. Clearly, the first passage time $\tau (0)$ must be one of the
ladder times $T_{i}$. Moreover, the ladder steps $Z_{m+1}-Z_{m}$ are
i.i.d., and by exercise 6, p. 232 of \cite{spitzer}, the random
variable $Z_{1}$ has finite absolute $r$th moment if the step
distribution $F_{RW}$ has finite absolute $(r+1)$th moment.  The key
observation is that for any 
$a \leq 0$ and any $x\geq 1$,
\begin{align*}
	P^{x}(W_{\tau (0)} \leq a)&= \sum_{k=1}^{x} G (x;k) P^{k}\{Z_{1}\leq  a-k\}\\
	&\leq \sum_{k=1}^{x}  P^{k}\{Z_{1}\leq  a-k\}\\
	&=\sum_{k=1}^{x}  P^{0}\{Z_{1}\leq  a-k\}
\end{align*}
where $G (x;k)$ is the probability under $P^{x}$ that the random walk
$W_{n}$ will visit the site $k$ at one of the ladder times
$T_{i}$. Thus, if $S$ has the limiting
overshoot distribution, then for any $a\leq 0$,
$$P(S \leq a) \leq \sum_{k=1}^{\infty} P(Z_1 \leq a-k).$$ 
This inequality together with the earlier observation about moments of
the ladder variable $Z_{1}$ implies that if $E^{0}|W_{1}|^{r}<\infty$
then $S$ has at least $r-2$ moments. 
\end{proof}

\begin{lem}\label{lemma:overshoot}
If the step distribution $F_{RW}$ satisfies the hypotheses of
Theorem~\ref{mainresultgeneral} then along any sequence
$x=x_{k}\rightarrow \infty$ such that \eqref{eq:subseq-lims} holds
uniformly on compact sets,
\begin{equation} \label{wts}
\lim_{k \rightarrow \infty} E^{y/\sqrt{u(x_{k})}}\left (
\frac{u(W_{\tau_0}+x_{k})}{u(x_{k})}\right ) = 1. 
\end{equation}
\end{lem}

\begin{proof}
As in the proof of Proposition~\ref{phicont} we will omit the
subscript $k$ on $x_{k}$ and write $\lim_{x \rightarrow \infty}$ to
mean convergence along the subsequence $x_{k}$. We also write $z=y/\sqrt{u (x)}$.
By Proposition \ref{phicont},  
\[
	\lim_{y \rightarrow 0} \lim_{x \rightarrow \infty} \frac{u
	 (x+z )}{u(x)} = 1,
\]
which implies that for any $\alpha>0$,
\begin{equation*}
\lim_{x \rightarrow \infty} \frac{u  (x+u (x)^{-\frac{1}{2}+\alpha}) }{u(x)} = 1.
\end{equation*}
By the monotonicity
of $u$,
\begin{align*}
	1 & \leq E^{z} \left ( \frac{u(W_{\tau_0}+x)}{u(x)} \right )\\
	  & = 	E^{z} \left ( \frac{u(W_{\tau_0}+x)}{u(x)} \right
	  )\mathbf{1}_{A} +E^{z} \left ( \frac{u(W_{\tau_0}+x)}{u(x)}
	  \right )\mathbf{1}_{A^{c}} \\
	  &=I+II,
\end{align*}
where
\begin{equation*}
	A= A (x)=\{W_{\tau (0)}\geq -u (x)^{-1/2+\alpha } \}.
\end{equation*}
By Chebyshev's inequality, for any $r>2$,
\[
	P^{x} (A^{c})\leq E|W_{\tau (0)}|^{r-2}u (x)^{(r-2)
	(\frac{1}{2} -\alpha)},
\] 
and by Lemma~\ref{moments}, if the step
distribution $F_{RW}$ has finite $r$th moment then $E|W_{\tau
(0)}|^{r-2}<\infty$. Since $u (u(W_{\tau_0}+x))/u (x)\leq u (x)^{-1}$,
it follows that 
\[
	II\leq C (u(x))^{(\frac{1}{2}-\alpha)(r-2)-1}.
\]

By hypothesis the step distribution $F_{RW}$ has finite $r$th
moment for some $r>4$, so the constant $\alpha >0$ can be chosen so that
the exponent in the last displayed inequality is
positive. Consequently, quantity $II$ converges to $0$ as $x \rightarrow
\infty$. On the other hand, $\lim_{x \rightarrow \infty}P^{x} (A)=1$,
and on the event $A$ the integrand in quantity $I$ is bounded by
\[
	\frac{u (x-u (x)^{-1/2+\alpha})}{u (x)} \longrightarrow 1,
\]
and so quantity $I$ converges to $1$ as $x \rightarrow \infty$.
\end{proof}

\begin{proof}[Proof of Proposition  \ref{phieq}]
Once again write $z=y/\sqrt{u (x)}$.
According to Corollary~\ref{corollary:fk}  and the Optional Stopping Theorem,
\begin{equation}\label{eq:chain}
\begin{aligned}
& \frac{u \bigl (x+{y}/{\sqrt{u(x)}} \bigr )}{u(x)}\\
= \; & E^{x+z}  \prod_{j=1}^{\tau_{x}} \bigl (1-H(u(W_j)) \bigr ) \frac{u(W_{\tau_x})}{u(x)} \\
= \; & E^{z}  \prod_{j=1}^{\tau_{0}} \bigl (1-H(u(W_j+x)) \bigr ) \frac{u(W_{\tau_0}+x)}{u(x)} \\
= \; & E^{z}  \exp \bigl \{ \sum_{j=1}^{\tau_{0}} \log (1-H(u(W_j+x)) \bigr \} \frac{u(W_{\tau_0}+x)}{u(x)} \\
= \; & E^{z}  \exp \bigl \{ \sum_{j=1}^{\tau_{0}} \bigl ( -\frac{\sigma^2}{2} u(W_j+x) + O(u(x)^2) \bigr ) \bigr \} \frac{u(W_{\tau_0}+x)}{u(x)} \\
= \; & E^{z}  \exp \bigl \{ -\frac{\sigma^2}{2} \sum_{j=1}^{\tau_{0}}  u(W_j+x) + \tau_0 O(u(x)^2) \bigr \} \frac{u(W_{\tau_0}+x)}{u(x)} \\
\end{aligned}
\end{equation}
The error term $O(u(x)^2)$ is bounded in magnitude by $Cu (x)^{2}$ for
some finite constant $C$ not depending on $x$, by virtue of our
standing hypothesis that the offspring distribution $F_{GW}$ has
finite third  moment,  which ensures that $H (u)$ has finite third
derivative at $u=0$. 

The invariance principle implies that as $x \rightarrow \infty$ the
distribution of the process $\sqrt{u(x)} W_{t/u(x)}/\eta $ under
$P^{z}$ converges weakly to that of a standard Brownian motion $B_{t}$
started at $B_{0}=y/\eta$. Consequently, the distribution of the renormalized first
passage time  $u (x)\tau_{0}$ converges weakly to that of
$\tau^{BM}_{0}$, and hence the error term $\tau_{0}O (u (x)^{2})$
converges in distribution  to $0$. Moreover,  along any sequence
$x=x_{k}\rightarrow \infty$ such that the convergence
\eqref{eq:subseq-lims} holds uniformly for $y$ in compact intervals,
\[
	\sum_{j=1}^{\tau_{0}}u (W_{j}+x)=u (x)\sum_{j=1}^{\tau_{0}}u (W_{j}+x)/u (x)
			        \stackrel{\mathcal{D}}{\longrightarrow}
			       \int_{0}^{\tau^{BM}_{0}} \phi (\eta B_{t})\,dt.
\]
Therefore, by Lemma~\ref{lemma:overshoot},  
\begin{multline*}
	\lim_{k \rightarrow \infty} E^{z}  \exp \bigl \{
	-\frac{\sigma^2}{2} \sum_{j=1}^{\tau_{0}}  u(W_j+x_{k}) + \tau_0
	O(u(x_{k})^2) \bigr \} \frac{u(W_{\tau_0}+x_{k})}{u(x_{k})}\\
	= E^{y/\eta}  \exp\left\{ -\frac{\sigma^2}{2}\int_0^{\tau_0^{BM}}
	\phi(\eta B_t) \,\mathrm{d}t \right\}.   
\end{multline*}
This together with the convergence \eqref{eq:subseq-lims} and the
chain of equalities \eqref{eq:chain} proves that $\phi$ must satisfy
the Feynman-Kac formula \eqref{phieqeq}.
\end{proof}

\begin{cor}\label{corollary:fk:de}
Under the hypotheses of Proposition~\ref{phieq}, 
\begin{equation}\label{eq:theLim}
	\lim_{x \rightarrow \infty} \frac{u (x+y/\sqrt{u (x)})}{u (x)}
	= \left(\frac{\sigma y}{\sqrt{6}\eta}+1 \right)^{-2}:=\phi (y)
\end{equation}
uniformly for $y\geq 0$.
\end{cor}

\begin{proof}
Since $u$ and $\varphi$ are monotone and bounded,
it suffices to prove that there is only one possible subsequential
limit function \eqref{eq:subseq-lims}, and that this limit is the
solution of the differential equation 
\begin{equation}\label{eq:fk-eqn}
	\phi '' (y)= {\sigma^{2}\phi (y)^{2}}/{\eta^{2}}
\end{equation}
that satisfies $\phi (0)=1$ and $\lim_{y \rightarrow \infty}\phi
(y)=0$.  But this follows from the Feynman-Kac representation
\eqref{phieqeq} of subsequential limits and Kac's theorem (cf., for
instance, \cite{ito-mckean}, sec. ~2.6), which implies that for
any positive, bounded, continuous function $V:[0,\infty ) \rightarrow
\zz{R}$ the function
\[
	\psi (y)=E^{y/\eta }\exp \left\{-\frac{\sigma^{2}}{2}\int_{0}^{\tau^{BM}_{0}}V (\eta
	B_{t})\,dt \right\} 
\]
is the unique bounded solution of the differential equation $\psi
''=\sigma^{2}V\psi /\eta^{2}$ satisfying $\psi (0)=1$.
\end{proof}

\subsection{Proof of Theorem~\ref{mainresultgeneral}}

To complete the proof of Theorem~\ref{mainresultgeneral} we must show
that 
\begin{equation}\label{eq:Objective}
	\lim_{x \rightarrow \infty}w (x) =1/\beta^{2}:=6\eta^{2}/\sigma^{2}	
\quad \text{where} \quad 	w (x):=x^{2}u (x).
\end{equation}
We will deduce \eqref{eq:Objective} from the asymptotic scaling law
\eqref{eq:theLim}, which in terms of the function $w$ may be rewritten
as
\begin{equation}\label{eq:w-fund}
	\lim_{x \rightarrow \infty}\frac{(1+\beta
	y)^{2}}{(1+y/\sqrt{w (x)})^{2}}\cdot 
	\frac{w (x (1+y/\sqrt{w (x)}))}{w (x)}=1.
\end{equation}
This  relation holds uniformly for $y\in [0,A]$, for any $A<\infty$, by
Corollary~\ref{corollary:fk:de}, since \eqref{eq:w-fund} is obtained
from  \eqref{eq:theLim} by multiplying both sides of \eqref{eq:theLim}
by  $(1+\beta y)^{2}$.

The function $w$ is not continuous, since $u$ has jump discontinuities
at the positive integers. However, it is ``asymptotically continuous''
in the sense that 
\[
	\lim_{x \rightarrow \infty} \sup_{0\leq y\leq 1} \left|\frac{w
	(x+y)}{w (x)} -1\right|=0.
\]
This follows directly from the corresponding assertion for the
function $u$, which in turn follows from the monotonicity of $u$ and
Corollary~\ref{corollary:fk:de}. Consequently, $w$ obeys the following
weak form of the intermediate value theorem.

\medskip 
\begin{lem}\label{lemma:noSkips}
Suppose there exist constants $0\leq A < B\leq \infty$ and  positive
integers $x_{1}<z_{1}<x_{2}<z_{2}<\dotsb $ such that $\lim w(x_{n})=A$
and $\lim w (z_{n})=B$. Then for every $C\in [A,B]$ there exist
integers $y_{n}\in [x_{n},z_{n}]$ and $y_{n}'\in [z_{n},x_{n+1}]$ such
that 
\[
	\lim w (y_{n})= \lim w (y_{n}')=C.
\]
\end{lem}

\begin{proof}
[Proof of Theorem~\ref{mainresultgeneral}] Denote by $0\leq A\leq
B\leq \infty$ the liminf and limsup of $w (x)$ as $x \rightarrow
\infty$. We must show that $A=B=1/\beta^{2}$. 

First we show that $A<\infty$  and $B>0$. Assume to the contrary that
$A=\infty$; then $w (x) \rightarrow \infty$ as $x \rightarrow \infty$,
and so there must be a sequence $x_{k} \rightarrow \infty$ such
that $w$ exits the interval $(0, w (x_{k}))$  at
$x_{k}$, that is, such that $w (x)\geq w (x_{k})$ for all
$x>x_{k}$.  The asymptotic scaling
relation \eqref{eq:w-fund}, with $x=x_{k}$, implies that for any
$y>0$,  
\[
	\lim_{k \rightarrow \infty}\frac{(1+\beta
	y)^{2}}{(1+y/w (x_{k}))^{2}}\cdot 
	\frac{w (x_{k} (1+y/\sqrt{w (x_{k})}))}{w (x_{k})}=1.
\]
Since $w (x_{k})\rightarrow \infty$, the first ratio converges to
$(1+\beta y)^{2}>1$; but the second ratio  is also at least $1$, by the
choice of $x_{k}$, so we have a contradiction. Therefore,
$A<\infty$. A similar argument shows that $B>0$.

Suppose next that $\infty >A>1/\beta^{2}$. Since $A=\liminf w (x)$ there
exists a sequence $x_{k}\rightarrow \infty$ such that
$w(x_{k})\rightarrow A$.  For any $y>0$ the  scaling
relation \eqref{eq:w-fund}, with $x=x_{k}$, implies
\[
	\lim_{k \rightarrow \infty}\frac{(1+\beta
	y)^{2}}{(1+y/\sqrt{A})^{2}}\cdot 
	\frac{w (x_{k} (1+y/\sqrt{A}))}{A}=1.
\]
But since $y>0$, it follows that
\[
	\lim_{k \rightarrow \infty} w (x_{k} (1+y/\sqrt{A}))) = A
	\frac{(1+y/\sqrt{A})^{2}}{(1+\beta y)^{2}}<A,
\]
contradicting the supposition that $A=\liminf w (x)$. Thus, $A\leq
1/\beta^{2}$. A similar argument proves that $B\geq
1/\beta^{2}$. Thus, $\beta^{-2}\in [A,B]$.

It remains to prove that $A=B$. If not, then it must be the case that
$\beta^{-2}<B$ or $\beta^{-2}>A$. Suppose that 
$A<\beta^{-2}$, and let $A^{*}\in (A,\beta^{-2})$; then the weak
intermediate value property (Lemma~\ref{lemma:noSkips}) implies that
there exist sequences $z_{n}\rightarrow \infty$
 and $x_{n}\rightarrow \infty$ such that 
\begin{align*}
	\lim_{n \rightarrow \infty} w (z_{n})&=A^{*};\\
	\lim_{n \rightarrow \infty} w (x_{n})&=A; \quad \text{and}\\
	w (x) \leq w (& z_{n}) \quad \text{for all}\; x\in [z_{n},x_{n}].
\end{align*}
Relation \eqref{eq:w-fund}, this time with $x=z_{n}$,
implies that uniformly for $y\in [0,C]$, for any $C<\infty$,
\[
	\lim_{n \rightarrow \infty} \frac{(1+\beta
	y)^{2}}{(1+y/\sqrt{A^{*}})^{2}}\cdot 
	\frac{w (z_{n} (1+y/\sqrt{A^{*}}))}{w (z_{n})}=1.
\]
But  since $A^{*}<\beta^{-2}$, 
\[
	\frac{(1+\beta	y)^{2}}{(1+y/\sqrt{A^{*}})^{2}}\leq 1
\]
for all
$y\geq 0$, with strict inequality except at $y=0$, and so 
\[
	\liminf \frac{w (z_{n} (1+y/\sqrt{A^{*}}))}{w (z_{n})}\geq 1
\]
uniformly on any interval $y\in [0,C]$, with strict inequality on any
sub-interval $y\in [\varepsilon ,C]$. This contradicts the hypothesis
that $w (x)\leq w (z_{n})$ for all $x\in [z_{n},x_{n}]$. Therefore,
$A=\beta^{-2}$. A similar argument shows that $B=\beta^{-2}$.

\end{proof}

\section{Conditional Limit Theorem}\label{sec:time-dependent}

\subsection{Space-time Feynman-Kac formula}\label{ssec:td-fk} Assume
throughout this section that $M_{n}$ is the rightmost particle
location in the $n$th generation of a branching random walk satisfying
the hypotheses of Theorem~\ref{mainresultgeneral}. (On the event that
there are no particles in the $n$th generation, set $M_{n}=-\infty$.)  The
\emph{unconditional} distribution function of the random variable
$M_{n}$ satisfies a time-dependent nonlinear convolution equation
similar to the time-independent equation \eqref{fund} satisfied by the
distribution of the maximal displacement random variable
$M$. Specifically, if
\begin{equation}\label{eq:mn-cdf}
	 v_{n} (x)=v (n,x):=P\{M_{n} > x\},
\end{equation}
then for every $n\geq 1$,
\begin{equation}\label{eq:nonlinear}
	v_{n} (x) =\sum_{y\in \zz{Z}} a_{k}Q (v_{n-1} (x-y)),
\end{equation}
where $1-Q (1-s)$ is the probability generating function of the
offspring distribution (see equation \eqref{eq:definition-Q}). The
objective of this section is to analyze the asymptotic behavior of
$v$, and in particular to show that $nv (n,[x\sqrt{n}])$ converges as
$n \rightarrow \infty$, for any $x\in \zz{R}$, to a 
distribution function that depends only on the variances of the
offspring and step distributions of the branching random walk.
The strategy will once again be to represent the solution of
\eqref{eq:nonlinear} by a discrete Feynman-Kac formula, and then to
show that after an appropriate rescaling  the Feynman-Kac expectations
converge to the corresponding  Feynman-Kac expectations for Brownian motion. As in
sec.~\ref{sec:main-result}, denote by $W_{n}$ a random walk with step
distribution \eqref{eq:rw-law}, and  by $P^{x}$ the law of the
random walk with initial point $W_{0}=x$. Then essentially the same arguments
as in the time-independent case prove the following assertion.

\begin{prop}\label{proposition:td-fk}
For each $n$ the process 
\begin{equation}\label{eq:space-time-mg}
	Z^{(n)}_{k}= v_{n-k} (W_{k})\prod_{j=1}^{k-1} (1-H (v_{n-j},
	\quad \text{for}\; k=0,1,2,\dotsc ,n 
	(W_{j}))) 
\end{equation}
is a martingale under $P^{x}$. Consequently, for any $n\leq m$,
\begin{equation}\label{eq:space-time-fk}
	v_{m} (x)=E^{x}v_{n} (W_{m-n})\prod_{j=1}^{m-n-1} (1-H (v_{m-j}
	(W_{j}))). 
\end{equation}
\end{prop}

\subsection{Monotonicity, tightness, and scaling limits}\label{ssec:monotonicity}

For each $n\geq 0$ the function $v_{n} (x)$ is non-increasing in $x$,
with limit $0$ as $x \rightarrow \infty$ and limit $v_{n}
(-\infty)=P\{\zeta \geq n \}$ as $x \rightarrow -\infty$, where
$\zeta$ is the extinction time of the branching random walk. By
Kolmogorov's theorem on the lifetime of a critical Galton-Watson
process (see \cite{athreya-ney}, ch.~1),  as $n \rightarrow \infty$,
\begin{equation}\label{eq:kolmogorov}
	v_{n} (-\infty)=P\{\zeta \geq n \}\sim \frac{2}{\sigma^{2}n}. 
\end{equation}
Hence, the functions $nv_{n} (x)$ are uniformly bounded.

\begin{lem}\label{lemma:tightness}
Under the hypotheses of Theorem~\ref{mainresultgeneral},
the family of rescaled distribution functions $ nv_{n} (x\sqrt{n})$ is
tight, that is, 
\begin{equation}\label{eq:tightness}
	\lim_{ x \rightarrow \infty} \sup_{n\geq 1} nP\{M_{n} \geq
	x\sqrt{n}\}=0 \quad \text{and} \quad 
	\lim_{ x \rightarrow -\infty} \sup_{n\geq 1} nP\{-\infty <M_{n} \leq
	x\sqrt{n}\}=0.
\end{equation}
\end{lem}

\begin{proof}
The first of these follows directly from Theorem
\ref{mainresultgeneral}, because $M_{n}\leq M$. The second follows
from Theorem \ref{mainresultgeneral} by reflection of the branching
random walk in the origin.
\end{proof}

\begin{remark}\label{remark:equicontinuity}
The hypothesis that the step distribution $F_{RW}$ of the branching
random walk has finite $4+\varepsilon$ moment is used here in an
essential way. If $F_{RW}$ has infinite $4-\varepsilon$ moment for
some $\varepsilon >0$ then Lemma~\ref{lemma:tightness} need not be
true: see the discussion at the end of sec.~\ref{sec:main-result}.
\end{remark}

\begin{lem}\label{lemma:uniform-equicontinuity}
For any $\varepsilon >0$ there exists $\delta >0$ such that if
$|x-y|\leq \delta \sqrt{n}$ and $n\leq m\leq n (1+\delta)$ then
\begin{equation}\label{eq:unif-equi}
	\left| \frac{v_{m} (y)}{v_{n} (x)}-1\right| \leq \varepsilon.
\end{equation}
\end{lem}

\begin{proof}
This follows by an argument similar to the proof of
Proposition~\ref{phicont}, using the Feynman-Kac formula
\eqref{eq:space-time-fk} and Donsker's invariance principle, since
$\sup_{x}v_{n} (x) \rightarrow 0$ as $n \rightarrow \infty$ and $H
(u)\sim \sigma^{2}u/2$ as $u \rightarrow 0$.
\end{proof}

\begin{cor}\label{corollary:uniform-equicontinuity}
Any sequence of positive integers has a subsequence $n_{k}\rightarrow
\infty$ along which the functions $(t,x)\mapsto nv_{[nt]} (x\sqrt{n})$
converge uniformly for $1\leq t\leq A$ and $x\in [-\infty ,\infty]$,
for any $A<\infty$. The set of possible limit functions $\varphi
(t,x)$ is compact in $C ([1,A]\times \bar{\zz{R}}$, and for each $t$
the function $\varphi (t,x)$ is non-increasing in $x$, with
\begin{equation}\label{eq:phi-lims}
	\lim_{x \rightarrow \infty}\varphi (t,x)=0
	\quad \text{and} \quad 
	\lim_{x \rightarrow -\infty} t\varphi (t,x)=2/\sigma^{2}.
\end{equation}
\end{cor}

\begin{proof}
All of the assertions except the limits \eqref{eq:phi-lims} follow
from  Lemma~\ref{lemma:uniform-equicontinuity}. The limit relations
\eqref{eq:phi-lims} follow from the tightness of the family $nv_{n}
(x\sqrt{n})$ and Kolmogorov's theorem \eqref{eq:kolmogorov}.
\end{proof}

\begin{cor}\label{corollary:brownian-scaling}
Let $\varphi (t,x)$ be any subsequential limit of the functions
$(t,x)\mapsto nv_{[nt]} (x\sqrt{n})$, for $1\leq t$ and $x\in [-\infty
,\infty]$. Then $\varphi$ satisfies the identity
\begin{equation}\label{eq:brownian-scaling-identity}
	\varphi (t+1,x)=E^{x/\eta} \varphi (1,\eta B_{t})\exp
	\left\{-\frac{\sigma^{2}}{2}\int_{0}^{t} \varphi (t+1-s,\eta B_{s}) \,ds  \right\},
\end{equation}
where under $P^{y}$ the process $B_{t}$ is a standard Brownian motion
started at $y$. Consequently, $\varphi$ satisfies the partial
differential equation 
\begin{equation}\label{eq:ivp}
	\frac{\partial \varphi}{\partial
	t}=\frac{\eta^{2}}{2}\frac{\partial^{2}\varphi}{\partial
	x^{2}} -\sigma^{2}\varphi^{2}
	\quad \text{for}\;\; t>1 \quad \text{and}\;\; x\in \zz{R}.
\end{equation}
\end{cor}

\begin{proof}
The integral representation \eqref{eq:brownian-scaling-identity}
follows from the discrete Feynman-Kac formula \eqref{eq:space-time-fk}
and the invariance principle by virtually the same argument as in the
proof of Proposition~\ref{phieq}. The differential equation
\eqref{eq:ivp} follows from the integral representation by Kac's
theorem. 
\end{proof}

It should be noted that the only \emph{a priori} bounds on the
functions $v_{n}$ needed to deduce the existence of subsequential limits
$\varphi (t,x)$ are those in Lemmas
\ref{lemma:tightness}--\ref{lemma:uniform-equicontinuity} and
Corollary~\ref{corollary:uniform-equicontinuity}. These use only the
crude estimate $M_{n}\leq M:=\max_{n\geq 1}M_{n}$ and the
results concerning the tail behavior of the distribution of $M$ proved
in section~\ref{sec:fk}. To prove that there is only one possible
subsequential limit function $\varphi (t,x)$ we will need the
following stronger \emph{a priori} bounds on  the
functions $v_{n}$.

\begin{lem}\label{lemma:strong-a-priori}
\begin{align}\label{eq:strong-apriori-plus}
	\lim_{x \rightarrow \infty} \sup_{n\geq 1}\, & nx^{2}v_{n}
	(x\sqrt{n})=0 \quad \text{and}\\
\label{eq:strong-apriori-minus}
	\lim_{x \rightarrow -\infty} \inf_{n\geq 1}\, & nx^{2} (1- v_{n}
	(x\sqrt{n}))=0.
\end{align}
\end{lem}

The proof is deferred to section~\ref{ssec:strong-a-priori} below.

\subsection{Super-Brownian motion and solutions of
\eqref{eq:ivp}}\label{ssec:spacetime-scaling} At first sight it might
appear that Corollary \ref{corollary:brownian-scaling} sheds no light
at all on the question of uniqueness of scaling limits, because in
general the solution to the partial differential equation
\eqref{eq:ivp} will depend on the initial condition $\varphi (1,x)$.
However, we will show, using the \emph{a priori} bounds in
Lemma~\ref{lemma:strong-a-priori}, that in fact Corollary
\ref{corollary:brownian-scaling} implies that there can be only one
scaling limit, and thereby complete the proof of
Theorem~\ref{theorem:conditional}. The key is the fact that solutions
of \eqref{eq:ivp} determine -- and are determined by -- the law of
super-Brownian motion, by the following \emph{duality formula}. For
ease of exposition, assume henceforth that $\eta^{2}=1$. (There is no
loss of generality in this, because solutions of \eqref{eq:ivp} can be
rescaled.)
 
\begin{prop}\label{proposition:scaling-uniqueness}
Let $X^{\delta_{x}}_{t}$ be a super-Brownian motion with branching parameter
$\sigma^{2}$ and initial mass distribution $X^{\delta_{x}}_{0}=\delta_{x}$, a
unit point mass at location $x\in \zz{R}$. Let $\varphi (t+1,x)$ be
the solution of the evolution equation \eqref{eq:ivp} with initial
condition $\varphi (1,x)=\varphi (x)$. Then
\begin{equation}\label{eq:duality}
	\varphi  (t+1,x)=-\log E \exp \{- \xclass{X^{\delta_{x}}_{t},\varphi }\}.
\end{equation}
Here $\xclass{\mu ,\varphi}$ denotes the integral of $\varphi$ against
the measure $\mu$.
\end{prop}

\begin{proof}
See, for instance, \cite{etheridge},
Cor.~1.25.
\end{proof}

Exploitation of the duality formula \eqref{eq:duality} will require
some elementary properties of super-Brownian motion.  First,
super-Brownian motion is equivariant under spatial translation:
$X^{\delta_{x}}_{t}$ can be obtained from $X^{\delta_{0}}_{t}$ by translating each of
the random measures $X^{\delta_{0}}_{t}$ by $x$ to the right. Second,
super-Brownian motion satisfies a \emph{scaling law}: if
$\tilde{X}_{t}$ is a super-Brownian motion with initial mass
distribution $\tilde{X}_{0}=A\delta_{x\sqrt{A}}$ then the rescaled
measure-valued process $X_{t}$ defined by
\[
	\xclass{X_{t},f}=\xclass{A^{-1}\tilde{X}_{At}, f (\sqrt{A}\,\cdot)}
\]
is a super-Brownian motion with initial mass distribution
$X_{0}=\delta_{x}$. Third, super-Brownian motion is \emph{infinitely
divisible}, in the following sense: the superposition of $m$
independent super-Brownian motions with initial mass distributions
$\nu_{i}$ is a super-Brownian motion with initial mass distribution
$\sum_{i=1}^{m}\nu_{i}$. This implies that the total mass $|X_{t}|$ at
time $t$ evolves as \emph{Feller diffusion}, and so $P\{X_{t}\not =0
\}\sim C/t$ as $t \rightarrow \infty$ with $C=2/\sigma^{2}$, where
$\sigma^{2}$ is the branching parameter (cf. \cite{legall}, sec. 2.1,
Theorem 1 and Example (iii)). It follows, by
\eqref{eq:duality}, that any solution $\varphi (t+1,x)$ of
\eqref{eq:ivp} whose initial condition $\varphi (x)=\varphi (1,x)$ is
non-increasing in $x$ and satisfies the boundary conditions
\eqref{eq:phi-lims}.  Finally, the scaling and infinite divisibility
properties imply that super-Brownian motion conditioned to live for
time at least $t$ has a limit law.

\begin{prop}\label{proposition:conditionedSBM}
If $X_{t}$ is a super-Brownian
motion with initial state $X_{0}=\delta_{0}$ then for any bounded,
continuous test function $f:\zz{R} \rightarrow \zz{R}$,
\begin{equation}\label{eq:conditioned-sbm}
	\lim_{t \rightarrow \infty}\mathcal{D} (t^{-1}\xclass{X_{t},f
	(\sqrt{t} \,\cdot)}\,|\, X_{t}\not =0) =\mathcal{D} (\xclass{Y^x_{1},f})
\end{equation}
where $Y_{t}=Y^0_{t}$ is the super-process specified in
Corollary~\ref{corollary:G}.
\end{prop}

\begin{proof}
See \cite{roelly-coppoletta-rouault} and \cite{evans-perkins}.
\end{proof}

\bigskip \noindent 
\textbf{Construction of the limit
process.}
The existence of the weak limit \eqref{eq:conditioned-sbm} is implicit
in the ``Poisson cluster'' representation of super-Brownian motion
(cf. \cite{etheridge}, ch. 1) but in fact weak convergence will not by
itself suffice for the arguments to follow, so we now sketch a
construction of super-Brownian motion in which the random measure
$Y^{x}_{1}$ arises naturally as an almost sure limit. First, observe
that infinite divisibility implies that a 
super-Brownian motion $X_{t}$ with initial state $X_{0}=\delta_{x}$
can be decomposed as
\[
	X_{t}=X'_{t}+X''_{t},
\]
where $X'_{t}$ and $X''_{t}$ are independent super-Brownian motions
started from initial measures $X'_{0}=X''_{0}=\delta_{x}/2$.  This
decomposition process can be iterated, so by standard arguments there
exist, on some probability space, countably many independent
super-Brownian motions $X_t^{n,m}$, where $m=0,1,2,\dotsc$ and
$n=1,2,\cdots,2^m$, such that for every pair $(n,m)$ the process
$X_t^{n,m}$ has initial state $X^{n,m}_{0}=\delta_{x}/2^{m}$ and
\[
	X_t^{n,m} = X_t^{n,m+1} + X_t^{n+2^m,m+1} \quad
	\text{for} \;\; t\geq 0.
\]
By the scaling property, the probability that any one of the $m$th
generation processes $\{X_t^{n,m} \}_{t\geq 0}$ survives to time $1$
is the same as the probability that the $0$th-generation super-Brownian motion
$\{X_t^{1,0} \}_{t\geq 0}$ survives to time $2^{m}$, which is $\sim
C/2^{m}$, where $C=2/\sigma^{2}$. Hence, if $F_m = \{ n \mid
X_1^{n,m} \neq 0 \}$ then $|F_{m}|$ converges in law to the Poisson
distribution with mean $C$. By construction, the random variables
$|F_{m}|$ are nondecreasing in $m$, so it follows that in fact the
sequence $\{|F_{m}| \}_{m\geq 0}$ is eventually constant, with Poisson
limit $N$.  By re-indexing the super-Brownian motions in each
generation $m$, we can arrange that for all sufficiently large $m$ the
processes $\{X^{n,m}_{t} \}_{t\geq 0}$ satisfy
\[
	X^{n,m}_{1}=X^{n,m+1}_{1} \quad \text{for all}\;\; 1\leq n\leq N,
\]
and only those processes  $\{X^{(n;m)}_{t} \}_{t\geq 0}$ for which
$n\leq N$ survive to time $t=1$. Conditional on the value of $N$, each
of the random measures 
\[
	X^{n,\infty}_{1}:= \lim_{m \rightarrow \infty} X^{n,m}_{1}
\]
is an independent version of $Y^{x}_{1}$.

\qed

\subsection{Asymptotic behavior of scaling limits}\label{ssec:asyScaling}

\begin{cor}\label{corollary:sbm}
Let $\varphi (t,x)$ be any subsequential limit of the functions
$(t,x)\mapsto nv_{[nt]} (x\sqrt{n})$, for $1\leq t$ and $x\in
[-\infty,\infty]$. Then for any
$x\in \zz{R}$,
\begin{equation}\label{eq:solution-asymptotics}
	\lim_{m \rightarrow \infty} 2^{m} \varphi
	(2^{m}+1,x2^{m/2})=\frac{2}{\sigma^{2}} 
	P\{Y^{x}_{1} (-\infty ,0]\not = 0 \}:=2G (x)/\sigma^{2}.
\end{equation}
Furthermore, for each $x\in \zz{R}$ the convergence
\eqref{eq:solution-asymptotics} holds uniformly over the set of all
possible subsequential limits $\varphi (t,x)$.
\end{cor}

\begin{proof}
Write $y=\sqrt{t}x$, and abbreviate $\varphi (1,x)=\varphi (x)$.
Since $P\{X^{\delta_y}_{t}\not =0 \}\sim
2/ (\sigma^{2}t)$ as $t \rightarrow\infty$, the duality formula
\eqref{eq:duality} implies that
\begin{align*}
	t\varphi (t+1,y)&=-t\log E\exp \{-\xclass{X^{\delta_y}_{t},\varphi } \}\\
		 &\sim tE (1-\exp \{-\xclass{X^{\delta_y}_{t},\varphi } \})\\
		 &=tE (1-\exp \{-\xclass{X^{\delta_y}_{t},\varphi }
		 \})\mathbf{1}\{X^{\delta_y}_{t}\not =0 \} \\ 
		 &\sim (2/\sigma^{2})E (( 1-\exp
		 \{-\xclass{X^{\delta_y}_{t},\varphi } \}) \,|\,
		 X^{\delta_y}_{t}\not 	 =0 ).\\
\end{align*}
The last expectation can be rewritten using the scaling property of
super-Brownian motion: 
\begin{align*}
	& E (( 1-\exp \{-\xclass{X^{\delta_y}_{t},\varphi } \}) \,|\, X^{\delta_y}_{t}\not =0 )\\
	= \; & 1- E ( \exp \{ -\int_{-\infty}^{\infty} \varphi(z) \,\mathrm{d}(X^{\delta_y}_{t}(z)) \} \,|\, X^{\delta_y}_{t}\not =0 )\\
	= \; & 1- E ( \exp \{ -t \int_{-\infty}^{\infty} \varphi(\sqrt{t}z) \,\mathrm{d}(t^{-1} X^{\delta_y}_{t}(\sqrt{t}z)) \} \,|\, X^{\delta_y}_{t}\not =0 )\\
	= \; & 1- E ( \exp \{ -t \int_{-\infty}^{\infty} \varphi(\sqrt{t}z) \,\mathrm{d}(X^{\delta_x/t}_{1}(z)) \} \,|\, X^{\delta_x/t}_{1}\not =0 ).
\end{align*}
By the construction sketched in  the preceding  subsection, versions
of the super-Brownian motions $X^{\delta_{x}/t}$, for $t=2^{m}$,
conditioned to survive to time $t=2^{m}$ can be constructed on a
common probability space along with a version of the random measure
$Y^{x}_{1}$ in such a way that $X^{\delta_{x}/2^{m}}_{1}=Y^{x}_{1}$
for all large $m$.  Hence, as $t \rightarrow \infty$ through powers of $2$,
\begin{align*}
	& \lim_{t \rightarrow \infty} E ( \exp \{ -t \int_{-\infty}^{\infty} \varphi(\sqrt{t}z) \,\mathrm{d}(X^{\delta_x/t}_{1}(z)) \} \,|\, X^{\delta_x/t}_{1}\not =0 )\\
	= \; & \lim_{t \rightarrow \infty} E  \exp \{ -t \int_{-\infty}^{\infty} \varphi(\sqrt{t}z) \,\mathrm{d}(Y^x_{1}(z)) \} .
\end{align*}

The result now follows from Lemma~\ref{lemma:strong-a-priori} and the
dominated convergence theorem. To see this, observe that the
exponential in the last expectation is bounded above by $1$, because
the function $\varphi$ is nonnegative. By
Lemma~\ref{lemma:strong-a-priori}, for each $z>0$
\begin{align*}
	\lim_{t \rightarrow \infty}  t\varphi (\sqrt{t}z)&=0 \quad \text{and}\\
	\lim_{t \rightarrow \infty}  t\varphi (-\sqrt{t}z)&=\infty 
\end{align*}
\emph{uniformly} over the set of all possible subsequential limits
$\varphi (y)=\lim nv_{n} (\sqrt{n}y)$.  Therefore, by dominated
convergence, as $t \rightarrow \infty$ through powers of $2$,
\[
	\lim_{t \rightarrow \infty} E  \exp \{ -t
	\int_{-\infty}^{\infty} \varphi(\sqrt{t}z)
	\,\mathrm{d}(Y^x_{1}(z)) \} =
	P\{\text{supp} (Y^{x}_{1})\subset (0,\infty) \}.
\]

\end{proof}

\subsection{Proofs of Theorem~\ref{theorem:conditional} and
Corollary~\ref{corollary:G}}\label{ssec:proof-thm-2} It suffices to
show that 
\[
\lim_{n \rightarrow \infty }nv_{n} (x\sqrt{n})=V (x):=\frac{2}{\sigma^{2}}G (x),
\]
where $G (x)$ is defined by \eqref{eq:definition-G}.  By
Corollary~\ref{corollary:uniform-equicontinuity}, subsequential limits
exist, and by Corollary~\ref{corollary:brownian-scaling} subsequential
limits must satisfy the partial differential equation
\eqref{eq:ivp}. What must be shown is that the only possible limit is
the function $V$.

Suppose then that there is a subsequence $n_{k}\rightarrow
\infty$ along which $nv_{n} (x\sqrt{n})\rightarrow U (x)$ for some function $U$. By
Corollary~\ref{corollary:uniform-equicontinuity} the sequence $n_{k}$
has a subsequence $n_{j} \rightarrow \infty$ along which the functions
$(t,x)\mapsto nv_{[nt]} (x\sqrt{n})$ converge uniformly for $1\leq
t\leq A$ and $x\in [-\infty ,\infty]$, for any $A<\infty$.  By
rescaling time, we can extract yet another subsequence $n_{i}$ along
which  the functions $(t,x)\mapsto nv_{[nt]} (x\sqrt{n})$ converge
uniformly for $t\in [2^{-m},2^{m}]$ and $x\in [-\infty ,\infty]$, for
any $m\geq 1$. Denote the limit function by $\varphi (t,x)$. 

By construction, $\varphi (1,x)=U (x)$; moreover, each section
$\varphi (2^{-m},x)$ is a subsequential limit of the functions
$x\mapsto 2^{-m}nv_{2^{-m}n} (x\sqrt{n}2^{m/2})$. Consequently,
Corollary~\ref{corollary:sbm} implies that $U=V$.

\qed

\subsection{Proof of Lemma~\ref{lemma:strong-a-priori}}\label{ssec:strong-a-priori}

The proof will use Theorem~\ref{mainresultgeneral} and the
Dawson-Watanabe theorem. Denote by $\xi^{1}_{t},\xi^{2}_{t},\dotsc$ the
(counting) measure-valued processes associated with independent
branching random walks each started by a single particle at the origin
at time $0$, and each governed by the same offspring distribution and
step distribution, with variances $\eta^{2}$ and $\sigma^{2}$,
respectively. Set
\begin{equation}\label{eq:superposition}
	S^{n}_{t}=\sum_{i=1}^{n}\xi^{i}_{t}.
\end{equation}
Thus, the process $\{S^{n}_{t} \}_{t\geq 0}$ is a branching random
walk initiated by $n$ particles all located at the origin.  (We
view the branching random walks as continuous-time processes that are
constant on time intervals $[m,m+1)$, with jumps at integer times
$m$.) The Dawson-Watanabe theorem asserts that the process
$\{S^{n}_{t} \}_{t\geq 0}$,  after
rescaling, converges in law as $n \rightarrow \infty$ to
super-Brownian motion $X_{t}$ with initial mass distribution
$X_{0}=\delta_{0}$. In particular,
\begin{equation}\label{eq:dw}
	\frac{1}{n}S^{n}_{nt} (\sqrt{n}\cdot)
	 \Longrightarrow X_{t}
\end{equation}
where the weak convergence is in the Skorohod topology on the space of
cadlag measure-valued processes.  The limiting super-Brownian motion
has local branching rate $\eta^{2}$ and diffusion coefficient
$\sigma^{2}$.

Super-Brownian motion $X_{t}$ in one dimension has the property that
with probability one, for each $t>0$ the random measure $X_{t}$ is
absolutely continuous relative to Lebesgue measure, with jointly
continuous density $X (t,x)$. The density $X (t,x)$ is jointly
continuous except at $t=0$, and for each $t>0$ has compact support in
$x$; as $t \rightarrow 0$ the support contracts to the point
$0$. Super-Brownian motion dies out in finite time, so 
\[
	M^{X}:=\sup \{x\in \zz{R}\,:\, \int_{0}^{\infty}
	X_{t}[x,\infty )\, dt>0\} 
\]
is well-defined, measurable, and finite. Denote by $\tau_{x}$ the
infimal time that $X_{t}[x,\infty )>0$; then the event $M^{X}\geq x$
coincides a.s. with $\tau_{x}<\infty$. The path-continuity properties
of the density $X (t,x)$ imply that for any $\varepsilon >0$ and any
compact interval $[x_{1},x_{2}]$ not containing $0$ there exist
$0<\delta <\Delta <\infty$ such that for all $x\in [x_{1},x_{2}]$,
\begin{equation}\label{eq:delta-Delta}
	P (\delta <\tau_{x}<\Delta \,|\, \tau_{x}<\infty)>1-\varepsilon .
\end{equation}

\begin{prop}\label{proposition:sbmSupport} For any $x>0$,
\begin{equation}\label{eq:sbmMaxAsymptotics}
	P\{M^{X}\geq x \} =P\{\tau_{x}<\infty  \}=1-\exp \{-C/x^{2}
	\}\quad \text{where} \;\; 
	C= \frac{6\eta^{2}}{\sigma^{2}}.
\end{equation}
\end{prop}

\begin{proof}
By a theorem of Dynkin (see, e.g., \cite{etheridge}, Ch.~8) the function 
\[
	 u(x)=-\log P\{M^{X}<x \}
\]
is the unique solution of the differential equation
$u''=\eta^{2}u^{2}/\sigma^{2}$ with boundary conditions $u (0)=\infty$
and $u (\infty)=0$.
\end{proof}

The crucial point of Proposition \ref{proposition:sbmSupport} is that
the distribution \eqref{eq:sbmMaxAsymptotics} coincides with the limit
distribution in
Corollary~\ref{corollary:max-displacement-superposition}. As noted
earlier, the weak convergence \eqref{eq:dw} does not by itself imply
that the maximal displacement $M^{n}$ of the branching random walk $S^{n}_{t}$
converges weakly (after rescaling) to $M^{X}$, because in the
Dawson-Watanabe scaling \eqref{eq:dw} individual particles receive
vanishingly small mass as $n \rightarrow \infty$, but it does imply
that
\[
	\liminf_{n \rightarrow \infty} P\{M^{n}\geq \sqrt{n}x \} \geq
	P \{M^{X}\geq x \}.
\]
Thus, Proposition~\ref{proposition:sbmSupport} and
Corollary~\ref{corollary:max-displacement-superposition} together
imply that the event $\{M^{n}\geq \sqrt{n}x \}$ is almost entirely
accounted for by sample evolutions in which a large number (order $n$)
of particles reach $\sqrt{n}x$. Furthermore,  by
\eqref{eq:delta-Delta}, they imply that for large $n$ the event
$\{M^{n}\geq \sqrt{n}x \}$ is mostly composed of sample evolutions in
which particles of the branching random walk reach $[\sqrt{n}x,\infty
)$ during the time interval $[n\delta ,n\Delta]$, where $0<\delta
<\Delta <\infty$ are as in \eqref{eq:delta-Delta}. Following is a
formal statement of this observation.

\begin{cor}\label{corollary:delta-Delta}
Denote by $M^{n}_{t}$ the location of the rightmost particle of the
branching random walk $S^{n}_{t}$ at time $t$. Then for any
$\varepsilon >0$ and any compact $K\subset (0,\infty)$ there exist
constants $0<\delta <\Delta <\infty$ such that for all $x\in K$ and
all sufficiently large $n$,
\begin{equation}\label{eq:delta-Delta-brw}
	nx^{2}P \{M^{n}_{tn}\geq \sqrt{n}x \;\;
	\text{for some}\;\; t\in [\delta ,\Delta] \} \geq  (1-\varepsilon) 
	(1-\exp \{-C /x^{2}\}).
\end{equation}
\end{cor}

\begin{proof}
[Proof of Lemma~\ref{lemma:strong-a-priori}] 
Recall that $v_{n} (y)$ is the probability that the maximal
displacement of the branching random walk $\xi^{1}$ exceeds $y$.
Since the reflection of this branching random walk is again a
driftless, critical branching random walk, the relation
\eqref{eq:strong-apriori-plus}  implies relation
\eqref{eq:strong-apriori-minus}, so it suffices to prove
\eqref{eq:strong-apriori-plus}, that is,
\begin{equation}\label{eq:pf-obj}
	\lim_{x \rightarrow \infty}\sup_{n\geq 1}nx^{2}P
	\{\xi^{1}_{n}[\sqrt{n}x,\infty )\geq 1 \}=0
\end{equation}
We will accomplish this by contradiction. Suppose that there
exist $\gamma >0$ and sequences $n_{k},x_{k}\rightarrow \infty$ along
which the left side of \eqref{eq:pf-obj} remains bounded below by
$\gamma$. Let $m=m_{k}=[n_{k}x_{k}^{2}]$ and $\theta
=\theta_{k}=1/x_{k}^{2}$, and consider the branching random
walk $S^{m}_{t}$ initiated by $m$ particles at  the origin; then our
hypothesis would imply
\begin{equation}\label{eq:devils-adv}
	P\{M^{m}_{m\theta }\geq \sqrt{m} \}\geq \gamma '=\frac{1}{2}
	(1-\exp \{-C/\gamma  \}) >0
\end{equation}
along the subsequence $m=m_{k}$. We will show that
\eqref{eq:devils-adv} is impossible.

Denote by $A^{m}_{\theta}$ the event that $M^{m}_{m\theta }\geq
\sqrt{m}$. We begin by observing that on this event the total number
of particles in the interval $[\beta \sqrt{m},\infty )$, for any fixed
$\beta >0$, must be small relative to $m$. This follows from the
Dawson-Watanabe  theorem: since $\theta =\theta_{k}
\rightarrow 0$, for any $\alpha >0$ the chance that the limiting
super-Brownian motion puts any mass on the interval $[\beta ,\infty )$ at
some time $t\leq \theta$ is vanishingly small, and so for any $\alpha
,\varepsilon >0$, if $k$ is sufficiently large then
\begin{equation}\label{eq:sparse-Plus}
	P\{S^{m}_{m\theta}[\beta \sqrt{m},\infty )\geq \alpha
	{m}\}<\varepsilon . 
\end{equation}
By reflection, it follows that 
\begin{equation}\label{eq:sparse-Minus}
	P\{S^{m}_{m\theta} (-\infty ,-\beta \sqrt{m}]\geq \alpha
	{m}\}<\varepsilon .
\end{equation}

Now fix $\varepsilon >0$ small, let $K=[1/2,2]$, and let $0<\delta
<\Delta<\infty$ be as in Corollary~\ref{corollary:delta-Delta}. Since
$\theta_{k}=x_{k}^{-2} \rightarrow 0$ as $k \rightarrow \infty$,
eventually $\theta_{k}<\delta /2$. By
\eqref{eq:sparse-Plus}--\eqref{eq:sparse-Minus}, with probability at
least $1-2\varepsilon$ there will be fewer than $\alpha m$ particles
outside the interval $[- \sqrt{m}/2, \sqrt{m}/2]$. Consider the
component branching random walks initiated by these particles at time
$\theta m$: if $\alpha \ll \delta /2$ then by Kolmogorov's theorem on
the extinction time of a critical Galton-Watson process, the
(conditional) probability (given the history of the branching random
walk up to generation $[\theta m]$) that one of these branching random
walks survives to time $\delta m$ is vanishingly small (as $\alpha
\rightarrow 0$); in particular, for suitable $\alpha$ this probability
will be $<\varepsilon$. Thus, except with probability not exceeding
$3\varepsilon$, on the event $A^{m}_{\theta}$ the only particles that
will survive to time $\delta m$ are descendants of particles located
in $[- \sqrt{m}/2, \sqrt{m}/2]$ at time $\theta m$.

Next, consider the total number $N^{m}_{m\theta}$ of particles in the
branching random walk at time $\theta m$. By Feller's theorem, the
processes $N^{m}_{mt}/m$ converge in law to a Feller diffusion $F_{t}$
started at $F_{0}=1$. Since diffusion processes have continuous paths,
and since $\theta =\theta_{k} \rightarrow 0$, it follows that with
probability approaching $1$, 
\[
	N^{m}_{m\theta} \stackrel{P}{ \longrightarrow}1
\]
as $k \rightarrow \infty$. In particular, the probability that
$N^{m}_{m\theta}\geq (1+\varepsilon)m$ will eventually be smaller than
$\varepsilon$.  

Finally, consider the event $F^{m}_{\delta ,\Delta}$ that
$M^{m}_{tm}\geq \sqrt{m}$ for some time $t\in [\delta ,\Delta]$. For
$F^{m}_{\delta ,\Delta}$ to occur, one of the particles alive at time
$m\theta$ must have a descendant that reaches the interval
$[\sqrt{m},\infty )$ after time $\delta m$. The probability that such
a particle is not located in $[-\sqrt{m}/2,\sqrt{m}/2]$ is less than
$3\varepsilon$. Moreover, in order that a particle located in
$[-\sqrt{m}/2,\sqrt{m}/2]$ at time $\theta m$ has a descendant in
$[\sqrt{m},\infty )$, the intervening trajectory must travel a
distance at least $\sqrt{m}/2$, and by Theorem~\ref{mainresultgeneral}
the (asymptotic) probability of this is less than $4C/m$, where
$C=6\eta^{2}/\sigma^{2}$. But the  probability that the number of particles
in $[-\sqrt{m}/2,\sqrt{m}/2]$ at time $\theta m$ exceeds
$(1+\varepsilon)m$ is less than $\varepsilon$, and hence, for large $k$,
\[
	P (F^{m}_{\delta ,\Delta}\,|\,A^{m}_{\theta})\leq 
	4\varepsilon + 1- (1-4C/m)^{(1+\varepsilon)m} \approx 
	4\varepsilon +1-\exp \{-4C (1+\varepsilon) \}.
\]
It follows (provided $\varepsilon >0$ is sufficiently small) that 
\[
	P ((F^{m}_{\delta ,\Delta })^{c}\cap A^{m}_{\theta})
	\geq (1-e^{-8C}) P (A^{m}_{\theta}).
\]
Since the event $M^{m}\geq \sqrt{m}$ contains $F^{m}_{\delta
,\Delta}\cup A^{m}_{\theta}$, we now have, by
Corollary~\ref{corollary:max-displacement-superposition}, 
\begin{align*}
	P\{M^{m}\geq \sqrt{m} \}&\sim 1-e^{-C} \\
		     &\geq P(F^{m}_{\delta ,\Delta})+P
		     ((F^{m}_{\delta ,\Delta })^{c}\cap
		     A^{m}_{\theta}) \\
		     &\geq (1-\varepsilon) (1-e^{-C})+ (1-e^{-8C})\gamma '
\end{align*}
Since $\varepsilon >0$ can be chosen arbitrarily small, this is
impossible, so we have arrived at a contradiction.

\end{proof}

\bibliographystyle{plain}
\bibliography{main}

\end{document}